\documentclass[12pt]{amsart}
\usepackage{amssymb,amsthm, times}
\usepackage{delarray,verbatim}
\usepackage{graphicx}
\usepackage{color}
\usepackage[T1]{fontenc}
\usepackage{epsfig}
\usepackage{bm}
\usepackage{ifpdf}
\usepackage{graphics}
\usepackage{amsmath}
\usepackage{amsfonts}
\usepackage{amssymb}
\usepackage[hmargin=0.8in,height=8.6in]{geometry}

\providecommand{\U}[1]{\protect\rule{.1in}{.1in}}
\newtheorem{remark}{Remark}
\newtheorem{theorem}{Theorem}

\newtheorem{proposition}{Proposition}
\newtheorem{definition}{Definition}
\newtheorem{lemma}{Lemma}
\newtheorem*{lemma*}{Lemma}

\begin{document}
\title{On the non-commutative fractional Wishart process.}
\author[J. C. Pardo]{Juan Carlos Pardo}
\address[Juan Carlos Pardo]{Department of Probability and Statistics, Centro
de Investigaci\'on en Matem\'aticas, Apartado Postal 402, Guanajuato GTO
36000, Mexico.}
\email{jcpardo@cimat.mx}
\author[J. L. P\'erez]{Jos\'e-Luis P\'erez}
\address[Jos\'e-Luis P\'erez]{Department of Probability and Statistics,
IIMAS-UNAM, Mexico City, Mexico.}
\email{garmendia@sigma.iimas.unam.mx }
\author[V. P\'{e}rez-Abreu]{Victor P\'{e}rez-Abreu}
\address[Victor P\'{e}rez-Abreu]{Department of Probability and Statistics,
Centro de Investigaci\'on en Matem\'aticas, Apartado Postal 402, Guanajuato
GTO 36000, Mexico.}
\email{pabreu@cimat.mx}
\thanks{This version: \today. }
\date{}
\maketitle

\begin{abstract}
We investigate the process of eigenvalues of a fractional Wishart process
defined by $N=B^*B$, where $B$ is the matrix fractional Brownian motion
recently studied in \cite{AN}. Using stochastic calculus with respect to the
Young integral we show that, with probability one, the eigenvalues do not
collide at any time. When the matrix process $B$ has entries given by
independent fractional Brownian motions with Hurst parameter $H\in(1/2,1)$,
we derive a stochastic differential equation in the Malliavin calculus sense
for the eigenvalues of the corresponding fractional Wishart process.
Finally, a functional limit theorem for the empirical measure-valued process
of eigenvalues of a fractional Wishart process is obtained. The limit is
characterized and referred to as the \textit{non-commutative fractional Wishart process}%
, which constitutes the family of fractional dilations of the free Poisson
distribution. \newline
\noindent {\small {\textbf{Key words and phrases:} Fractional Wishart matrix
process, measure valued process, Young integral,
fractional calculus.} }
\end{abstract}

\section{Introduction.}

In this paper, we make a systematic study of the dynamics and the limiting
non-commutative distribution of the eigenvalue process of a fractional
Wishart matrix process. More specifically, let $H\in (0,1),n,p\geq 1$ and $%
B=\{\{b_{ij}(t),t\geq 0\},1\leq i\leq p,1\leq j\leq n\}$ be a set of $%
p\times n$ independent one-dimensional fractional Brownian motions with the
same Hurst parameter $H.$ That is, each $b_{ij}$ is a zero mean Gaussian
process with covariance 
\begin{equation*}
\mathbb{E}\Big[b_{ij}(t)b_{ij}(s)\Big]=\frac{1}{2}\left(
t^{2H}+s^{2H}-|t-s|^{2H}\right) .
\end{equation*}%
As in \cite{AN}, we introduce $(N(t),t\geq 0)$, the matrix fractional
Brownian motion process with parameter $H$ whose components satisfy $%
N_{ij}(t)=b_{ij}(t)$, for $t\geq 0$.

A fractional\ Wishart process is the nonnegative definite $n\times n$ matrix
process defined by $X= N ^{\ast}N,$ where $N^{\ast}$ denotes the transpose
of some matrix $N.$ Let $\left( \lambda_{1}(t),\lambda_{2}(t),..,\lambda
_{n}(t), t\ge 0\right)$ be the $n$-dimensional stochastic process of
eigenvalues of $X$ and consider the empirical spectral process of the
eigenvalues $\lambda_{1}^{(n)}(t)\geq\lambda_{2}^{(n)}(t)\geq\cdots\geq
\lambda_{n}^{(n)}(t)\geq0$ of $X^{(n)}=n^{-1}X$, i.e., 
\begin{equation}
\mu_{t}^{(n)}=\frac{1}{n}\sum_{j=1}^{n}\delta_{\lambda_{j}^{(n)}(t)},\qquad%
\text{$t\geq0$.}  \label{em}
\end{equation}

Different aspects of the dynamics and asymptotics of this spectral process
have been considered by several authors in the case $H=1/2$ of the classical
Wishart process. In this case, for $n\geq1$ fixed, Bru \cite{Bru} considered
the dynamics and non-colliding phenomena of the eigenvalue process, proving
that the spectral process is an $n$-dimensional diffusion given by the
system of non-smooth diffusion equations 
\begin{equation}
\lambda_{i}(t)=\lambda_{i}(0)+2\int_{0}^{t}\sqrt{\lambda_{i}(s)}\cdot d\nu
^{i}(s)+\int_{0}^{t}\left( p+\sum_{i\not =j}\frac{\lambda_{i}(s)+\lambda
_{j}(s)}{\lambda_{i}(s)-\lambda_{j}(s)}\right) ds,  \label{SDEWis}
\end{equation}
where $\nu^{i}$ are independent Brownian motions for $i=1,\dots,n$, and ``$%
\cdot$'' denotes the It\^{o} stochastic integral. Moreover, Bru \cite{Bru}
also showed that if $\lambda_{1}^{(n)}(0)\geq\dots\geq\lambda_{n}^{(n)}(0)$,
then a.s.\ the eigenvalues do not collide at any time, i.e., 
\begin{equation}
\mathbb{P}(\lambda_{1}(t)>\dots>\lambda_{n}(t),\forall t>0)=1.
\label{NonColl}
\end{equation}
The main tool for proving (\ref{SDEWis}) is It\^{o}'s formula for
matrix-valued semimartingales, and for (\ref{NonColl}), a McKean type
argument in the classical stochastic calculus.

Still in the classical case $H=1/2$, for fixed $t>0,$ the asymptotic
distribution of $\mu _{t}^{(n)}$ is given by the classical pioneering work
of Marchenko and Pastur, \cite{MP}. Namely, recall that the \textit{free
Poisson distribution} (or \textit{Marchenko--Pastur distribution}) $\mu
_{c}^{f,p}$, $c>0$ is the probability measure on $\mathbb{R}_{+}$ defined by 
\begin{equation}
\mu _{c}^{f,p}(dx)=%
\begin{cases}
\nu _{c}(dx), & \qquad \text{$c\geq 1$,}\notag \\ 
(1-c)\delta _{0}(dx)+\nu _{c}(dx), & \qquad \text{$c<1$},%
\end{cases}
\label{MPlaw}
\end{equation}%
where 
\begin{align}
\nu _{c}(dx)& =\frac{\sqrt{(x-a)(b-x)}}{2\pi x}1_{\{(a,b)\}}(x)dx,
\label{MPlaw1} \\
a& =(1-\sqrt{c})^{2},\quad b=(1+\sqrt{c})^{2}.  \notag
\end{align}%
It was shown in \cite{MP} that $\mu _{c}^{f,p}$ is the asymptotic spectral
distribution, when $t=1$, of the empirical spectral measure $\mu _{1}^{(n)}$
as $\lim_{n\rightarrow \infty }\frac{p}{n}=c>0$. For fixed positive $t\neq
1, $ the asymptotic spectral distribution of the empirical spectral measure $%
\mu _{t}^{(n)}$, as $\lim_{n\rightarrow \infty }\frac{p}{n}=c>0,$ is the
family $(\mu _{c}(t),{t>0})$ of dilations of the free Poisson distribution
which is given by $\mu _{c}(t)=\mu _{c}^{fp}\circ h_{t}^{-1},$ \ where $%
h_{t}(x)=tx$ (see for instance Cavanal-Dubilard and Guionnet \cite{CG} and P%
\'{e}rez-Abreu and Tudor \cite{AT}).

The asymptotic behavior of the empirical spectral measure-valued process of (%
\ref{SDEWis}) falls into the framework of \ the study of limiting
measure-valued processes of interacting diffusion particles governed by It%
\^{o} stochastic diffusion equations with strong interactions, as an
eigenvalue process of a matrix diffusion having the property that the
particles never collide. The general aim in this framework is to show that
the empirical spectral measure process converges weakly in the space of
continuous probability measure-valued processes to a deterministic law. This
general direction of study was considered by Cepa and Lepingle \cite{cele},
Chan \cite{cha}, and Rogers and Shi \cite{rosh}, among others, in the case
of some Gaussian matrix diffusions, turning out to become limiting
non-commutative processes as a free Brownian motion. See also \cite{KO01}, 
\cite{KT04}, \cite{KT13}, \cite{AT0}, and references therein.

In this tendency, for the case $H=1/2$ of the Wishart process$,$ it was
proved in \cite{CG} and \cite{AT} that the empirical spectral process $\{(
\mu_{t}^{(n)}, {t\geq0}); n\ge1\}$ converges weakly in the space of
continuous probability measure-valued processes to the family $(\mu_{c}(t), {%
t\ge 0})$ of dilations of the free Poisson distribution. The proof of this
result is mainly based on an appropriate It\^{o}'s formula for matrix-valued
semimartingales and large deviations or classical It\^{o} calculus
inequalities estimates.

The aim of the present paper is to make a systematic study of the spectra of
the fractional matrix Wishart process with Hurst parameter $H\in (1/2,1)$,
and understand several properties such as the dynamics of its eigenvalue
process, its noncollision properties, and the limiting family of the
corresponding empirical spectral measure-valued processes. Since fractional
Brownian motion is not a semimartingale, the main tools we use are based on
the Skorokhod and Young stochastic calculus. Recently, Nualart and P\'{e}%
rez-Abreu \cite{AN} considered the dynamics and noncollision property of the
eigenvalues of a symmetric matrix fractional Brownian motion, and in \cite%
{PPP}, there was derived the functional limit of the corresponding empirical
spectral measure-valued processes, which is the non-commutative fractional
Brownian motion considered by Nourdin and Taqqu \cite{NT}.

The final goal of the present paper is to find the non-commutative limit
process of the empirical spectral measure-valued processes $\{(\mu
_{t}^{(n)},t\geq 0);n\geq 1\}$ of the fractional Wishart matrix process of
Hurst parameter $H\in (1/2,1)$, as $n$ goes to infinity. Specifically, in
this paper we introduce the non-commutative fractional Wishart process of
Hurst parameter $H\in \left[ 1/2,1\right) $ as the family $(\mu _{c,H}(t),{%
t>0})$ of fractional dilations of the free Poisson distribution given by $%
\mu _{c,H}(t)=\mu _{c}^{f,p}\circ (h_{t}^{H})^{-1}$, where $%
h_{t}^{H}=t^{2H}x $. That is, 
\begin{equation}
\mu _{c,H}(t)(dx)=%
\begin{cases}
\nu _{c}(t)(dx), & \qquad \text{$c\geq 1$,} \\ 
(1-c)\delta _{0}(dx)+\nu _{c}(t)(dx), & \qquad \text{$c<1$},%
\end{cases}
\label{fdfpl}
\end{equation}%
with $\mu _{c,H}(0)=\delta _{0}$. Then, as our main results, we prove the
following functional limit theorem for the empirical spectral measure-valued
processes $\{(\mu _{t}^{(n)},t\geq 0);n\geq 1\}$. \ Let $\Pr (\mathbb{R})\ $%
be the space of probability measures on $\mathbb{R}$ endowed with the
topology of weak convergence and let $C\left( \mathbb{R}_{+},\Pr (\mathbb{R}%
)\right) $ be the space of continuous functions from $\mathbb{R}_{+}\ $into$%
\ \Pr (\mathbb{R}),$ endowed with the topology of uniform convergence on
compact intervals of $\mathbb{R}_{+}.$\ {\color{blue} }

\begin{theorem}
\label{thcfw} Let $H\in (1/2,1)$ and $(\lambda _{1}^{(n)}(t)\geq \lambda
_{2}^{(n)}(t)\geq \cdots \geq \lambda _{n}^{(n)}(t)\geq 0,t\geq 0)$ be the
eigenvalue process of $X^{(n)},$ the scaled fractional\ Wishart process of
Hurst parameter $H.$ Assume that $\mu _{0}^{(n)}$ converges weakly to $%
\delta _{0}$, and that $\lim_{n\rightarrow \infty }\frac{p}{n}=c>0$. Then
the family of empirical spectral measure-valued processes $\{(\mu
_{t}^{(n)},t\geq 0);n\geq 1\}$ converges weakly in $C(\mathbb{R}_{+},\mathrm{%
Pr}(\mathbb{R}))$ to the unique continuous probability-measure valued
function corresponding to the law of the non-commutative fractional Wishart
process of Hurst parameter $H$ $\ (\mu _{c,H}(t),{t>0})$ described in (\ref%
{fdfpl}).
\end{theorem}

The strategy to prove this theorem is as follows, including some results
that are important on their own. We first consider the dynamics and
noncollision of the eigenvalues of a fractional Wishart process. The goal of
Section 2 is to derive a stochastic differential equation for the
eigenvalues of a fractional Wishart process in the framework of the
Skorokhod integral with respect to the multivariate fractional Brownian
motion. For preliminaries on the stochastic calculus with respect to
fractional Brownian motion, we refer to \cite{nose}, \cite{No}, and \cite{N}%
. We start with results on the first and second derivatives of the
eigenvalues of a nonnegative definite matrix $X=N^{\ast}N$ as functions of
the entries of the matrix $N.$ A detailed consideration of the derivatives
of the eigenvalues of an Hermitian matrix $X$, but in terms of the elements
of $X$, is carried out in Anderson\textit{\ et. al} \cite{A} and Tao \cite%
{Ta12}. Then we consider in Theorem \ref{Ito} a new It\^{o}'s formula for
the Skorokhod integral of functions related to the growth of the second
derivative of the eigenvalues of fractional Wishart processes, here denoted
by $X$. In Section 3 we prove the noncollision of the eigenvalues of $X$ at
any time. We follow closely the proof in the case of the fractional
symmetric matrix Brownian motion in \cite{AN}, using stochastic calculus
with respect to Young's integral as well as appropriate estimates for the
moments of the repulsion force of the eigenvalue processes and the joint
distribution of the eigenvalues of the fractional Wishart matrix.

The functional asymptotics of the empirical spectral measure-valued process
is considered in\ Section 4. We first apply our It\^{o}'s formula to find
appropriate expressions for the integrated processes 
\begin{equation*}
\langle \mu _{t}^{(n)},f\rangle =\frac{1}{n}\sum_{i=1}^{n}f(\lambda
_{i}^{(n)}(t)),\qquad t\geq 0.
\end{equation*}%
Then we prove tightness and the weak convergence of the family of measures $%
\{(\mu _{t}^{(n)},t\geq 0);n\geq 1\}$ in the space $C(\mathbb{R}_{+},\mathrm{%
Pr}(\mathbb{R}))$.
Finally, we characterize the family of laws $(\mu _{c,H}(t),t\geq 0)$ of a
non-commutative fractional Wishart process of Hurst parameter $H$ in terms
of the initial value problem for the corresponding Cauchy--Stieltjes
transform $G_{{c,H}}$ \ of $\mu _{c,H}$ (see Proposition \ref{IVP}).

\section{The stochastic differential equation for the eigenvalues.}

\subsection{Matrix Calculus and notation.}

In this section, we present some results on the eigenvalues of a nonnegative
definite symmetric matrix that will be needed during the course of this
paper. The computatios are similar to those explained in \cite{A} for the
Hermitian and in \cite{AN} for the symmetric case without the nonnegative
definite assumption.

We denote by $\mathcal{N}_{pn}$ the collection of $p\times n$ matrices. For
a matrix $N\in\mathcal{N}_{pn}$, we use the coordinates $N_{ij}$, with $%
1\leq i\leq p$, $1\leq j\leq n$, to denote the element on the $i$th row and
the $j$th column of $N$. For simplicity, we write $N=(N_{ij})$. Let $N^{*}$
denote the transpose of $N$. In order to work with Wishart matrices, we
define $X:=N^{*}N$, which is clearly symmetric and therefore belongs to $%
\mathcal{H}_{n}$, the space of symmetric $n$-dimensional matrices.

Let $\mathcal{U}_{n}^{vg}$ be the set of orthogonal matrices $U$ such that $%
U_{ii}>0$ for all $i$, $U_{ij}\not =0$ for all $i,j$, and all minors of $U$
have non-zero determinants. We denote by $\mathcal{N}^{vg}_{pn}$ the set of
matrices $N\in\mathcal{N}_{pn}$ such that there is a factorization 
\begin{equation*}
X:=N^{*}N=U\Lambda U^{*},
\end{equation*}
where $\Lambda$ is a diagonal matrix with entries $\lambda_{i}=\Lambda_{ii}$
such that $\lambda_{1}>\lambda_{2}>\dots>\lambda_{n}$ and $U\in\mathcal{U}%
_{n}^{vg}$. We also denote by $\mathcal{H}_{n}^{vg}$ the space of symmetric $%
n$-dimensional matrices $X$ such that $X=N^{*}N$ and $N \in\mathcal{N}%
^{vg}_{pn}$. The matrices in the set $\mathcal{N}^{vg}_{np}$ will be called 
\textit{very good} matrices, and we identify $\mathcal{N} ^{vg}_{np}$ with
an open subset of $\mathbb{R}^{np}$. Moreover, the complement of $\mathcal{N}%
^{vg}_{np}$ has zero Lebesgue measure.

Let $\mathcal{S}_{n}$ be the set 
\begin{equation*}
\mathcal{S}_{n}=\Big\{(\lambda _{1},\dots ,\lambda _{n})\in \mathbb{R}%
^{n}:\lambda _{1}>\lambda _{2}>\dots >\lambda _{n}\Big\}.
\end{equation*}%
For any $\lambda \in \mathcal{S}_{n}$, let $\Lambda ^{\lambda }$ be the
diagonal matrix such that $\Lambda _{ii}^{\lambda }=\lambda _{i}$. We
consider the mapping $T:\mathcal{U}_{n}^{vg}\rightarrow \mathbb{R}%
^{n(n-1)/2} $ defined as follows 
\begin{equation*}
T(U):=\left( \frac{U_{12}}{U_{11}},\dots ,\frac{U_{1p}}{U_{11}},\dots ,\frac{%
U_{n-1n}}{U_{n-1n-1}}\right) .
\end{equation*}%
It is known that $T$ is bijective and smooth, see \cite{A}, \cite{AN}. Next,
we introduce the mapping $\hat{T}:\mathcal{S}_{n}\times T(\mathcal{U}%
_{n}^{vg})\rightarrow \mathcal{H}_{n}^{vg}$ by $\hat{T}(\lambda
,Z)=T^{-1}(Z)\Lambda ^{\lambda }T^{-1}(Z)^{\ast }$ which turns out to be a
smooth bijection. We denote by $\hat{\Phi}$ the inverse of $\hat{T}$, i.e., $%
\hat{\Phi}(X):=(\lambda ,T(U))$, and observe that it can be defined as a
function of the associated \textit{very good} matrix of $X$: in other words,
we define a function $\Phi :\mathcal{N}_{pn}^{vg}\rightarrow 
\mathcal{S}_{n}\times T(\mathcal{U}_{n}^{vg})$ such that $\Phi (N):=\hat{\Phi%
}(X)$. As a consequence of these facts, it is clear that $\lambda (X)$ is a
smooth function of $N\in \mathcal{N}_{pn}^{vg}$.

Next, we suppose that $N$ is a smooth function of a parameter $\theta \in 
\mathbb{R}$. Then, 
\begin{equation}
\frac{\partial \lambda _{i}}{\partial \theta }=\left( U^{\ast }\frac{%
\partial X}{\partial \theta }U\right) _{ii}\qquad \text{and}\qquad \frac{%
\partial ^{2}\lambda _{i}}{\partial \theta ^{2}}=\left( U^{\ast }\frac{%
\partial ^{2}X}{\partial \theta ^{2}}U\right) _{ii}+2\sum_{j\not=i}\frac{%
|\left( U^{\ast }\frac{\partial X}{\partial \theta }U\right) _{ij}|^{2}}{%
\lambda _{i}-\lambda _{j}}.  \label{ev2a}
\end{equation}%
On the one hand, if we compute the values of the eigenvalues of $X$ in terms
of the entries of the matrix $N$, we observe 
\begin{equation}
\lambda
_{i}=\sum_{s=1}^{p}\sum_{r=1}^{n}\sum_{l=1}^{n}U_{ri}N_{sr}N_{sl}U_{li}=%
\sum_{s=1}^{p}\left( \sum_{r=1}^{n}U_{ri}N_{sr}\right) ^{2}.  \label{ev3}
\end{equation}%
On the other hand, if we take $\theta =N_{kh}$ we deduce 
\begin{equation}
\frac{\partial X_{rl}}{\partial N_{kh}}=\sum_{s=1}^{p}\left( \frac{\partial
N_{sr}}{\partial N_{kh}}N_{sl}+\frac{\partial N_{sl}}{\partial N_{kh}}%
N_{sr}\right) =N_{kl}\mathbf{1}_{\{r=h\}}+N_{kr}\mathbf{1}_{\{l=h\}},
\label{ev4}
\end{equation}%
Putting all the pieces together, we obtain 
\begin{equation*}
\frac{\partial \lambda _{i}}{\partial N_{kh}}=\sum_{r=1}^{n}%
\sum_{l=1}^{n}U_{ri}N_{kl}\mathbf{1}_{\{r=h\}}U_{li}+\sum_{r=1}^{n}%
\sum_{l=1}^{n}U_{ri}N_{kr}\mathbf{1}_{\{l=h\}}U_{li}=2U_{hi}%
\sum_{r=1}^{n}U_{ri}N_{kr}.
\end{equation*}%
Now, we are interested in computing the second derivative of the
eigenvalues. We start with the first term of (\ref{ev2a}) with $\theta
=N_{kh}$. Thus from (\ref{ev4}), we deduce 
\begin{equation*}
\frac{\partial ^{2}X_{rl}}{\partial N_{kh}^{2}}=2\mathbf{1}_{\{r=l=h\}},
\end{equation*}%
implying 
\begin{equation*}
\left( U^{\ast }\frac{\partial ^{2}X}{\partial N_{kh}^{2}}U\right)
_{ii}=2\sum_{r=1}^{n}\sum_{l=1}^{n}U_{ri}U_{li}\mathbf{1}_{\{r=l=h%
\}}=2U_{hi}^{2}.
\end{equation*}%
For the second term of (\ref{ev2a}), we use again (\ref{ev4}) and observe 
\begin{equation*}
\left( U^{\ast }\frac{\partial X}{\partial N_{kh}}U\right)
_{ij}=\sum_{l=1}^{n}U_{lj}U_{hi}N_{kl}+\sum_{r=1}^{n}U_{ri}U_{hj}N_{kr}.
\end{equation*}%
Therefore 
\begin{equation*}
\frac{\partial ^{2}\lambda _{i}}{\partial N_{kh}^{2}}=2U_{hi}^{2}+2\sum_{j%
\not=i}\frac{\left\vert
\sum_{l=1}^{n}U_{lj}U_{hi}N_{kl}+\sum_{r=1}^{n}U_{ri}U_{hj}N_{kr}\right\vert
^{2}}{\lambda _{i}-\lambda _{j}}.
\end{equation*}

Now, let us apply the above computations to the particular case of the 
\textit{fractional Wishart process} which is defined below. Let us consider
a family of independent fractional Brownian motions with Hurst parameter $%
H\in(1/2,1)$, $B=\{(b_{ij}(t),t\geq0), 1\leq i\leq p, 1\leq j\leq n\}$. As
in \cite{AN}, we introduce $(N(t), t\ge0)$, the matrix fractional Brownian
motion process with parameter $H$ whose components satisfy $%
N_{ij}(t)=b_{ij}(t)$, for $t\ge0$.

\begin{definition}
Let $(N(t), t\ge0)$ be the matrix fractional Brownian motion with parameter $%
H$. We call \textit{fractional Wishart process of order $n$ with parameter $%
H $} to the process $(X(t), t\ge0)$ satisfying $X(t)=N^{*}(t)N(t)$, for $%
t\ge0$.
\end{definition}

Following the previous discussion, for any $i\in\{1,\ldots, n\}$, we deduce
that there exists a function $\Phi_{i}:\mathbb{R}^{pn}\to\mathbb{R}$, which
is $C^{\infty}$ in an open subset $G\subset\mathbb{R}^{pn}$, with $G^{c}$
having Lebesgue measure 0, and such that $\lambda_{i}(t)=\Phi_{i}(N(t))$ for 
$t\ge0$. Therefore using the fact that $N_{kh}(t)=b_{kh}(t)$, we have 
\begin{equation}
\frac{\partial\Phi_{i}}{\partial b_{kh}}=2U_{hi}\sum_{r=1}^{n}U_{ri}b_{kr},
\label{ev5}
\end{equation}
and 
\begin{equation}  \label{sd2}
\begin{split}
\frac{\partial^{2}\Phi_{i}}{\partial b_{kh}^{2}} & =2U_{hi}^{2}+2\sum_{i\not
=j}\frac{|\sum_{l=1}^{n}U_{lj}U_{hi}b_{kl}+%
\sum_{r=1}^{n}U_{ri}U_{hj}b_{kr}|^{2}}{\lambda_{i}-\lambda_{j}} \\
& =2U_{hi}^{2}+2\sum_{i\not =j}\frac{\left( U_{hi}^{2}\left(
\sum_{l=1}^{n}U_{lj}b_{kl}\right) ^{2}+U_{hj}^{2}\left( \sum_{l=1}^{n}
U_{li}b_{kl}\right)
^{2}+2U_{hi}U_{hj}\sum_{l=1}^{n}U_{lj}b_{kl}\sum_{l=1}^{n}
U_{li}b_{kl}\right) }{\lambda_{i}-\lambda_{j}}.
\end{split}%
\end{equation}
On the other hand, we note 
\begin{align}
\sum_{k=1}^{p}\sum_{h=1}^{n}\frac{\partial^{2}\Phi_{i}}{\partial b_{kh}^{2}}
& =2\sum_{k=1}^{p}\sum_{h=1}^{n}U_{hi}^{2}+2\sum_{i\not =j}\frac{\sum
_{h=1}^{n}U_{hi}^{2}\sum_{k=1}^{p}\left( \sum_{l=1}^{n}U_{lj}b_{kl}\right)
^{2}+\sum_{h=1}^{n}U_{hj}^{2}\sum_{k=1}^{p}\left(
\sum_{l=1}^{n}U_{li}b_{kl}\right) ^{2}}{\lambda_{i}-\lambda_{j}}  \notag
\label{sd} \\
& +4\sum_{i\not =j}\frac{1}{\lambda_{i}-\lambda_{j}}%
\sum_{h=1}^{n}U_{hi}U_{hj}\sum_{k}\left(
\sum_{l=1}^{n}U_{lj}b_{kl}\sum_{l=1}^{n}U_{li}b_{kl}\right) \\
& =2\left( p+\sum_{i\not =j}\frac{\lambda_{i}+\lambda_{j}}{\lambda
_{i}-\lambda_{j}}\right) ,  \notag
\end{align}
where in the last identity we have used the fact that $U$ is an orthogonal
matrix and the identity (\ref{ev3}).

\subsection{Stochastic calculus for the fractional Brownian motion.}

In order to describe the evolution of the eigenvalues of a matrix fractional
Brownian motion we present a modification of Theorem 3.1 of \cite{AN}, which
is a multidimensional version of the It\^o formula for the Skorokhod
integral, in the case of functions that are smooth only on a dense open
subset of Euclidean space and satisfy some growth requirements. More
specifically, the modification is related to the growth of the second
derivative.

We refer to the monograph of Nualart \cite{N} for the definition of the
Skorokhod integral. For the definition of the space $\mathbb{L}^{1,p}_{H,i}$%
, for $p>1$ and $1\le i\le n$, we refer to section 2 of \cite{AN}.

\begin{theorem}
\label{Ito} Suppose $B^{H}$ is an $n$-dimensional fractional Brownian motion
with Hurst parameter $H>1/2$. Consider a function $F:\mathbb{R}^{n}\to%
\mathbb{R}$ such that:

\begin{itemize}
\item[(i)] There exists an open set $G\subset\mathbb{R}^{n}$ such that $%
G^{c} $ has zero Lebesgue measure and $F$ is twice continuously
differentiable in $G$.

\item[(ii)] $|F(x)|+\big|\frac{\partial F}{\partial x_{i}}\big|\leq
C(1+|x|^{M})$, for some constants $C>0$ and $M>0$ and for all $x\in G$ and $%
i=1,\dots,n$.

\item[(iii)] For each $i\in\{1,\dots,n\}$ and for each $s>0$ and $q\in[1,2)$, 
\begin{equation*}
\mathbb{E}\left[ \bigg|\frac{\partial^{2} F}{\partial x_{i}^{2}}(B^{H}_{s})%
\bigg|^{q}\right] \leq C,
\end{equation*}
for some constant $C>0$.
\end{itemize}

Then, for each $i=1,\dots,n$ and $t\in[0,T]$, the process $\{\frac{\partial F%
}{\partial x_{i}}(B^{H}_{s})\mathbf{1}_{[0,t]}(s),s\in[0,T]\}$ belongs to
the space $\mathbb{L}^{1,1/H}_{H,i}$ and 
\begin{equation}  \label{ito}
F(B^{H}_{t})=F(B^{H}_{0})+\sum_{i=1}^{n}\int_{0}^{t}\frac{\partial F}{%
\partial x_{i}}(B^{H}_{s})\delta B_{s}^{H,i}+H\sum_{i=1}^{n}\int_{0}^{t}%
\frac{\partial^{2} F}{\partial x_{i}^{2}}(B^{H}_{s})s^{2H-1}ds.
\end{equation}
\end{theorem}

\begin{proof}
We observe that the proof of this result employs similar arguments as those
used in the proof of Theorem 3.1 of \cite{AN}, with the exception of the
argument that verifies that Equation (\ref{ito}) is well defined.

To this end, it is enough to verify that the process $u_{i}(s)=\frac{%
\partial F}{\partial x_{i}}(B^{H}_{s})\mathbf{1}_{[0,t]}(s)$ belongs to the
space $\mathbb{L}^{1,1/H}_{H,i}$. Indeed, using conditions (ii) and (iii) we
have 
\begin{equation*}
\mathbb{E}\left[ \int_{0}^{T}|u_{i}(s)|^{1/H}ds\right] \leq C^{1/H}\mathbb{E}%
\left[ \int_{0}^{T}(1+|B^{H}_{s}|^{M})^{1/H}ds\right] <\infty.
\end{equation*}
and 
\begin{equation*}
\mathbb{E}\left[ \int_{0}^{T}\int_{0}^{T}|D_{r}^{(i)}u_{i}(s)|^{1/H}drds%
\right] =\mathbb{E}\left[ \int_{0}^{T}s\bigg|\frac{\partial^{2} F}{\partial
x_{i}^{2}}(B_{s}^{H})\bigg|^{1/H}ds\right] \leq\frac{C}{2}T^{2}.
\end{equation*}
On the other hand, taking $q=1$ in condition (iii), we also have for each $%
i\in\{1,\dots,n\}$, 
\begin{equation*}
\mathbb{E}\left[ \int_{0}^{t}\bigg|\frac{\partial^{2} F}{\partial x_{i}^{2}}%
\bigg|s^{2H-1}ds\right] <\infty,\quad\text{ for }\quad t>0.
\end{equation*}
given the fact that $H>1/2$. As a consequence, all the terms in (\ref{ito})
are well defined.
\end{proof}

\subsection{The SDE governing the eigenvalues of a fractional Wishart
process.}

In this section, we are interested in studying the dynamics of the
eigenvalues of a fractional Wishart process as governed by a stochastic
differential equation that depends on the Skorokhod integral.

Recall from Theorem 7.1.2 of \cite{HA} that the joint density of the
eigenvalues of $X(t)$ on $\mathcal{S}_{n}$, with respect to the Lebesgue
measure, satisfies 
\begin{equation}
c_{n,p}\prod_{j=1}^{n}\left( \lambda_{j}^{(p-n-1)/2}s^{-pnH}\exp\left( -%
\frac{\lambda_{j}}{2s^{2H}}\right) \right)
\prod_{j<k}|\lambda_{k}-\lambda_{j}|.  \label{evd}
\end{equation}

\begin{theorem}
Let $H\in(1/2,1)$ and $\{N(t),t\geq0\}$ be a matrix fractional Brownian
motion with parameter $H$ defined as above. Furthermore, let $N(0)$ be an
arbitrary deterministic $p\times n$ matrix. For each $t\geq0$, let $%
\lambda_{i},\dots,\lambda_{n}$ be the eigenvalues of the fractional Wishart
process of order $n$, $X=N^{*}N$. Then, for any $t>0$ and $i=1,\dots,n$, 
\begin{equation}  \label{edev}
\lambda_{i}(t)=\lambda_{i}(0)+\sum_{k=1}^{p}\sum_{h=1}^{n}\int_{0}^{t}\frac{%
\partial\Phi_{i}}{\partial b_{kh}}(N(s))\delta
b_{kh}(s)+2H\int_{0}^{t}\left( p+\sum_{i\not =j}\frac{\lambda_{i}(s)+%
\lambda_{j}(s)}{\lambda_{i}(s)-\lambda_{j}(s)}\right) s^{2H-1}ds.
\end{equation}
\end{theorem}

\begin{proof}
Without loss of generality, we may assume that $N(0)=0$. Now let us check
that $\{\lambda_{i}(t), i=1,\dots,n\}$ satisfies the conditions of Theorem %
\ref{Ito}. Using (\ref{ev3}) and the Cauchy--Schwarz inequality, we observe 
\begin{align}
\sum_{i=1}^{n}\Phi_{i}^{2}(N(t))=\sum_{i=1}^{n}\left( \sum_{l=1}^{p}\left(
\sum_{r=1}^{n}U_{ri}(t)b_{lr}(t)\right) ^{2}\right)
^{2}\leq\sum_{i=1}^{n}\left(
\sum_{l=1}^{p}\sum_{r=1}^{n}b_{lr}^{2}(t)\right) ^{2}\leq n||N(t)||_{2}^{4},
\notag
\end{align}
where $||\cdot||_{2}$ denotes the Euclidean norm of the columns of a matrix.
In particular, for each $1\le i\le n$, we have 
\begin{align}
|\Phi_{i}(N(t))|\leq\sqrt{n}||N(t)||_{2}^{2}.  \notag
\end{align}
On the other hand, using (\ref{ev5}) and the Cauchy--Schwarz inequality, we
obtain 
\begin{align}
\bigg|\frac{\partial\Phi_{i}}{\partial b_{kh}}(N(t))\bigg|%
^{2}=4U_{hi}^{2}(t)\left( \sum_{r=1}^{n}U_{ri}(t)b_{kr}(t)\right) ^{2}\leq4
||N(t)||_{2}^{2},  \notag
\end{align}
implying 
\begin{equation*}
\bigg|\frac{\partial\Phi_{i}}{\partial b_{kh}}(N(t))\bigg|\leq2||N(t)||_{2}.
\end{equation*}
Therefore, for each $1\le i\le n$, we have 
\begin{equation*}
|\Phi_{i}(N(t))|+\bigg|\frac{\partial\Phi_{i}(N(t))}{\partial b_{kh}}\bigg|%
\leq\sqrt{n}||N(t)||_{2}^{2}+2||N(t)||_{2}.
\end{equation*}
Now, observe that using (\ref{sd}) we can verify 
\begin{align}
\bigg|\frac{\partial^{2} \Phi_{i}}{\partial b_{kh}^{2}}(N(t)) & \bigg|%
\leq2U_{hi}^{2}(t)+2\sum_{i\not =j}\frac{%
|\sum_{l=1}^{n}U_{lj}(t)U_{hi}(t)b_{kl}(t)+%
\sum_{r=1}^{n}U_{ri}(t)U_{hj}(t)b_{kr}(t)|^{2}}{|\lambda_{i}(t)-%
\lambda_{j}(t)|}  \notag \\
& \leq\sum_{k=1}^{p}\sum_{h=1}^{n}\left( 2U_{hi}^{2}(t)+2\sum_{i\not =j}%
\frac{|\sum_{l=1}^{n}U_{lj}(t)U_{hi}(t)b_{kl}(t)+%
\sum_{r=1}^{n}U_{ri}(t)U_{hj}(t)b_{kr}(t)|^{2}}{|\lambda_{i}(t)-%
\lambda_{j}(t)|}\right)  \notag \\
& =2\left( p+\sum_{i\not =j}\frac{\lambda_{i}(t)+\lambda_{j}(t)}{%
|\lambda_{i}(t)-\lambda_{j}(t)|}\right) .  \notag
\end{align}
Hence using the joint density of the eigenvalues given by (\ref{evd}), the
previous equation, and Jensen's inequality, we obtain for $q\in[1,2]$ 
\begin{align*}
&\mathbb{E} \left[ \bigg|\frac{\partial^{2} \Phi_{i}}{\partial b_{kh}^{2}}%
(N(t))\bigg|^{q}\right] \leq2^{q}\mathbb{E}\left[ \left( p+\sum_{i\not =j}%
\frac{\lambda_{i}(t)+\lambda_{j}(t)}{|\lambda_{i}(t)-\lambda_{j}(t)|}\right)
^{q}\right] \leq4^{q}p^{q}+K_{n,p,q}\sum_{i\not =j}\mathbb{E}\left[ \frac{%
|\lambda_{i}(t)+\lambda_{j}(t)|^{q}}{|\lambda _{i}(t)-\lambda_{j}(t)|^{q}}%
\right] \\
& =4^{q }p^{q}+K_{n,p,q}\sum_{i\not =j}\int_{\mathcal{S}_{p}}\prod_{j=1}^{n}%
\left( \lambda_{j}^{(p-n-1)/2}t^{-npH}\exp\left( -\frac{\lambda_ {j}}{2t^{2H}%
}\right) \right) \prod_{j<k}|\lambda_{k}-\lambda_{j}|\frac{%
|\lambda_{i}+\lambda_{j}|^{q}}{|\lambda_{i}-\lambda_{j}|^{q}}d\gamma,
\end{align*}
where $\gamma$ denotes Lebesgue measure and $K_{n,p,q}$ is a positive
constant that only depends on $q,p,n$. Using the change of variables $\lambda_{i}=t^{2H}\mu_{i}$, we observe the last integral in
the previous inequality is a constant that only depends on $q,p,n$, thus 
\begin{equation}  \label{cee}
\mathbb{E} \left[ \bigg|\frac{\partial^{2} \Phi_{i}}{\partial b_{kh}^{2}}%
(N(t))\bigg|^{q}\right] \le\tilde{K}_{n,p,q},
\end{equation}
where $\tilde{K}_{n,p,q}$ is a positive constant that only depends on $q,p,n$%
. The result now follows using Theorem \ref{Ito} and (\ref{sd}).
\end{proof}

\begin{remark}
Let us consider the case $H=1/2$, therefore the process $(N(t), t\ge0)$
corresponds to a $p\times n$ matrix of independent Brownian motions. So we
have that Equation (\ref{edev}) takes the form 
\begin{equation*}
\lambda_{i}(t)=\lambda_{i}(0)+2\sum_{k=1}^{p}\sum_{h=1}^{n}\int_{0}^{t}\frac{%
\partial\Phi_{i}}{\partial b_{kh}}(N(s))\delta
b_{kh}(s)+2H\int_{0}^{t}\left( p+\sum_{i\not =j}\frac{\lambda_{i}(s)+%
\lambda_{j}(s)}{\lambda _{i}(s)-\lambda_{j}(s)}\right) ds.
\end{equation*}
Using the fact that the process $\displaystyle\frac{\partial\Phi_{i}}{%
\partial b_{kh}}(N(s))$ is adapted to the filtration generated by the
process $N$, then the Skorokhod integral coincides with the It\^{o}
integral, implying that the stochastic integral in (\ref{edev}) satisfies 
\begin{align}
\sum_{k=1}^{p}\sum_{h=1}^{n}\int_{0}^{t}\frac{\partial\Phi_{i}}{\partial
b_{kh}}(N(s))\delta b_{kh}(s) & =\sum_{k=1}^{p}\sum_{h=1}^{n}\int_{0}^{t}%
\frac{\partial\Phi_{i}}{\partial b_{kh}}(N(s))\cdot db_{kh}(s)=2\int_{0}^{t}%
\sqrt{\lambda_{i}}(s)\cdot dY^{i}(s),  \notag
\end{align}
where ``$\cdot$'' denotes the It\^o integral. Using (\ref{ev5}), $Y^{i}$
satisfies 
\begin{equation}
Y_{t}^{i}=\sum_{k=1}^{p}\sum_{h=1}^{n}\int_{0}^{t}U_{hi}(s)\frac{\sum
_{l=1}^{n}b_{kl}(s)U_{li}(s)}{\sqrt{\lambda_{i}(s)}}\cdot db_{kh}(s).  \notag
\end{equation}
Computing the quadratic variation of $Y^{i}$, we observe 
\begin{equation*}
\langle
Y^{i},Y^{i}\rangle_{t}=\sum_{k=1}^{p}\sum_{h=1}^{n}\int_{0}^{t}U_{hi}^{2}(s)%
\frac{\left( \sum_{l=1}^{n}b_{kl}(s)U_{li}(s)\right) ^{2}}{\lambda_{i}(s)}%
ds=t,
\end{equation*}
where in the last equality we used (\ref{ev3}) and (\ref{ev5}). Hence from L%
\'{e}vy's Characterization Theorem, we deduce that $Y^{i}$ is a Brownian
motion. Therefore Equation (\ref{edev}) takes the form 
\begin{equation*}
\lambda_{i}(t)=\lambda_{i}(0)+2\int_{0}^{t}\sqrt{\lambda_{i}(s)}\cdot d\nu
^{i}(s)+\int_{0}^{t}\left( p+\sum_{i\not =j}\frac{\lambda_{i}(s)+\lambda
_{j}(s)}{\lambda_{i}(s)-\lambda_{j}(s)}\right) ds,
\end{equation*}
where $\nu^{i}$ is a Brownian motion for $i=1,\dots,n$, which is the system
of SDE's obtained by Bru \cite{Bru}.
\end{remark}

\section{No collision of eigenvalues.}

In this section we show that, almost surely, the eigenvalues of the
fractional Wishart process do not collide. Recall that the fractional
Wishart process is defined as $X=N^{*}N$, where $N$ denotes the matrix
fractional Brownian motion with Hurst parameter $H>1/2$. Let us assume that $%
X(0)$ is a fixed deterministic symmetric matrix.

The following result follows closely the proof of Theorem 4.1 \cite{AN},
despite the fact that the fractional Wishart process is not Gaussian. For
the sake of completeness, we provide its proof.

\begin{theorem}
Denote by $\{(\lambda_{i}(t))_{t\geq0}, 1\le i\le n\}$ the eigenvalues of
the fractional Wishart process $(X(t), t\ge0)$. Assume that $%
\lambda_{1}(0)\geq\dots \geq\lambda_{n}(0)$. Then, 
\begin{equation*}
\mathbb{P}(\lambda_{1}(t)>\dots>\lambda_{n}(t),\forall t>0)=1.
\end{equation*}
\end{theorem}

\begin{proof}
We first assume that for fixed $t_{0}>0$, we have $\lambda_{1}(t_{0})>\dots>%
\lambda_{n}(t_{0})$. Since the matrix $X$ is symmetric, we can use the
Hoffman--Wielandt inequality (see \cite{HW}) to deduce the
following
\begin{align}  \label{dvcim}
\sum_{i=1}^{n}\Big(\lambda_{i}(t)-\lambda_{i}(s)\Big)^{2}\leq\sum_{i,j=1}^{n}%
\Big(X_{ij}(t)-X_{ij}(s)\Big)^{2} & =\sum_{i,j=1}^{n}\left| \sum_{k=1}^{p}%
\Big( b_{ki}(t)b_{kj}(t)-b_{ki}(s)b_{kj}(s)\Big) \right| ^{2} .  \notag
\end{align}
On the one hand, from the previous identity and applying the Jensen
inequality twice, we get for $r\ge2$, 
\begin{equation*}
\begin{split}
\Big|\lambda_{i}(t)-\lambda_{i}(s)\Big|^{r} & \le\left(
\sum_{i,j=1}^{n}\left| \sum_{k=1}^{p}\Big( %
b_{ki}(t)b_{kj}(t)-b_{ki}(s)b_{kj}(s)\Big) \right| ^{2}\right) ^{r/2} \\
& \le n^{r-2}\sum_{i,j=1}^{n}\left| \sum_{k=1}^{p}\Big( %
b_{ki}(t)b_{kj}(t)-b_{ki}(s)b_{kj}(s)\Big) \right| ^{r} \\
& \le n^{r-2}p^{r-1}\sum_{i,j=1}^{n} \sum_{k=1}^{p}\big| %
b_{ki}(t)b_{kj}(t)-b_{ki}(s)b_{kj}(s)\big|^{r}.
\end{split}%
\end{equation*}%
On the other hand, observe 
\begin{align}
\Big(b_{ki}(t)b_{kj}(t)-b_{ki}(s)b_{kj}(s)\Big)=b_{ki}(t)\Big(%
b_{kj}(t)-b_{kj}(s)\Big)+b_{kj}(s)\Big(b_{ki}(t)-b_{ki}(s)\Big).  \notag
\end{align}
Hence using the independence between $b_{ki}$ and $b_{kj}$ and the previous
identity, we get 
\begin{align}  \label{ineq3.1}
\mathbb{E}\Big[ \big|b_{ki}(t)b_{kj}(t)-b_{ki}(s)b_{kj}(s)\big|^{r}\Big] %
\leq C_{r}|t-s|^{rH}(t^{rH}+s^{rH}),
\end{align}
where $C_{r}$ is a positive constant that only depends on $r$. Putting all
the pieces together, we deduce that for any $T>0$ and $s,t\in[t_{0},T]$,
there exists a constant $C_{n,p,r,T}$ depending on $n,p,r, T$, such that 
\begin{equation}  \label{evda}
\mathbb{E}\Big[\big|\lambda_{i}(t)-\lambda_{i}(s)\big|^{r}\Big]\leq n^{r-2}
p^{r-1}\sum_{i,j=1}^{n}\sum_{k=1}^{p}\mathbb{E}\Big[ \big|%
b_{ki}(t)b_{kj}(t)-b_{ki}(s)b_{kj}(s)\big|^{r}\Big]\leq C_{n,p,r,T}
|t-s|^{rH}.
\end{equation}%
 Since $rH>1$, we deduce that the paths of $\lambda_{i}$ are H\"{o}lder
continuous of order $\beta$ for any $\beta<H$.

We now follow the same arguments as those used in Theorem 4.1 of \cite{AN}.
We consider the stopping time 
\begin{equation*}
\tau:=\inf\Big\{t>t_{0}:\lambda_{i}(t)=\lambda_{j}(t)\text{ for some $i\not
=j$}\Big\}.
\end{equation*}
Observe that $\tau>t_{0}$ a.s., and that on the random interval $%
[t_{0},\tau) $ the function $\log(\lambda_{i}(t)-\lambda_{j}(t)) $, where $%
i\not =j$, is well defined. Since the paths of $\lambda_{i}$ are H\"{o}lder
continuous for each $\beta<H$, we can apply the stochastic calculus with
respect to Young's integral and deduce that for any $t<\tau\wedge T$, 
\begin{equation}
\log(\lambda_{i}(t)-\lambda_{j}(t))=\log(\lambda_{i}(t_{0})-%
\lambda_{j}(t_{0}))+\int_{t_{0}}^{t}\frac{1}{\lambda_{i}(s)-\lambda_{j}(s)}%
d(\lambda _{i}(s)-\lambda_{j}(s)).
\end{equation}
Therefore for $1-H<\alpha<\frac{1}{2}$, we obtain 
\begin{equation*}
\int_{t_{0}}^{t}\frac{1}{\lambda_{i}(s)-\lambda_{j}(s)}d\lambda_{i}(s)=%
\int_{t_{0}}^{t}I_{i,j}(s)J_{j}(s)ds+\frac{\lambda_{i}(t)-\lambda _{i}(t_{0})%
}{\lambda_{i}(t_{0})-\lambda_{j}(t_{0})},
\end{equation*}
where 
\begin{align}
I_{i,j}(s):=D_{t_{0}+}^{\alpha}(\lambda_{i}-\lambda_{j})_{0}^{-1}(s) & =%
\frac{1}{\Gamma(1-\alpha)}\left[ s^{-\alpha}\bigg(\frac{1}{\lambda
_{i}(s)-\lambda_{j}(s)}-\frac{1}{\lambda_{i}(t_{0})-\lambda_{j}(t_{0})}%
\right)  \notag \\
& \hspace{2cm} +\alpha\int_{t_{0}}^{s}\frac{(\lambda_{i}(s)-\lambda
_{j}(s))^{-1}-(\lambda_{i}(y)-\lambda_{j}(y))^{-1}}{(s-y)^{\alpha+1}}dy\bigg]%
,  \notag
\end{align}
and 
\begin{equation*}
J_{j}(s):=D_{t-}^{1-\alpha}\lambda_{i,t-}(s)=\frac{1}{\Gamma(\alpha)}\left( 
\frac{\lambda_{i}(s)-\lambda_{i}(t)}{(t-s)^{1-\alpha}}+(1-\alpha)\int_{s}^{t}%
\frac{\lambda_{i}(s)-\lambda_{i}(y)}{(y-s)^{2-\alpha}}dy\right) .
\end{equation*}
We claim that 
\begin{equation}  \label{prob}
\mathbb{P}\left( \int_{t_{0}}^{t}|I_{i,j}(s)||J_{j}(s)|ds<\infty,\text{ for
all $t\geq t_{0}$, $i\not =j$}\right) =1.
\end{equation}
Before we prove the above identity, we first observe that H\"{o}lder's
inequality with exponents $\ell,q>1$ such that $1/\ell+1/q=1$, imply 
\begin{equation*}
\mathbb{E}\Big[|I_{i,j}(s)||J_{j}(s)|\Big]\leq\mathbb{E}\Big[%
|I_{i,j}(s)|^{\ell}\Big]^{1/\ell}\mathbb{E}\Big[|J_{j}(s)|^{q}\Big]^{1/q}.
\end{equation*}
From (\ref{evda}), we deduce that for any $\beta\in(1-\alpha, H)$, there
exists a r.v. $G$ with moments of all orders such that 
\begin{equation}  \label{evd2}
|\lambda_{i}(u)-\lambda_{i}(s)|\leq\mathbf{k}_{n,p,T}G|s-u|^{\beta},\qquad%
\text{for all}\quad i\in\{1,\ldots, n\},
\end{equation}
for all $s,u\in[t_{0},t]$ with $t\leq T$, and $\mathbf{k}_{n,p,T}$ is a
positive constant that only depends on $n,p$ and $T$. The above leads to the
estimate 
\begin{equation}  \label{est}
\mathbb{E}\Big[|J_{j}(s)|^{q}\Big]\leq\mathbf{k}_{n,p,T, q}\mathbb{E}\Big[%
G^{q}\Big],
\end{equation}
for all $q>1$ and for some constant $\mathbf{k}_{n,p,T, q}>0$. In order to
estimate $\mathbb{E}[|I_{i,j}(s)|^{\ell}]$, we consider the integral part in
the definition of $I_{i,j}$ and we denote it by 
\begin{equation*}
K_{i,j}(s):=\int_{t_{0}}^{s}\frac{(\lambda_{i}(s)-\lambda_{j}(s))^{-1}-(%
\lambda_{i}(y)-\lambda_{j}(y))^{-1}}{(s-y)^{\alpha+1}}dy.
\end{equation*}
Thus using the same estimates as in the proof of Theorem 4.1 of \cite{AN},
we  obtain that there exists a constant $K>0$ such that 
\begin{align}
\|K_{i,j}(s)\|_{\ell} & \leq K2^{a}\|G^{a}\|_{p_{1}}\int_{t_{0}}^{s}\Big(%
\||\lambda_{i}(y)-\lambda_{j}(y)|^{b-1}\|_{p_{2}}\||\lambda
_{i}(s)-\lambda_{j}(s)|^{-1}\|_{p_{3}}  \notag \\
& \hspace{3cm}+\||\lambda_{i}(y)-\lambda_{j}(y)|^{-1}\|_{p_{3}}\||\lambda
_{i}(s)-\lambda_{j}(s)|^{b-1}\|_{p_{2}}\Big)(s-y)^{a\beta-\alpha-1}dy,
\end{align}%
where $b=1-a$, $\frac{1}{\ell}=\frac{1}{p_{1}}+\frac{1}{%
p_{2}}+\frac{1}{p_{3}}$, with $p_{i}>1$ for $i=1,2,3$. We choose $%
a,p_{1},p_{2}$ and $p_{3}$ such that 
\begin{equation*}
a>\frac{\alpha}{\beta},\qquad p_{3}<2,\qquad p_{2}<\max\left\{%
\frac{2\beta}{\alpha}, \frac{2}{a}\right\},
\end{equation*}
which is possible by taking $\ell$ and $p_{1}$ close to $1$ and since $%
\alpha<\frac{1}{2}<\beta$. In order to prove identity (\ref{prob}), we need
to estimate 
\begin{equation*}
\mathbb{E}\Big[|\lambda_{i}(s)-\lambda_{j}(s)|^{-r}\Big],\qquad\text{ for }
\quad r<2.
\end{equation*}
Recall that the joint density of the eigenvalues $\lambda_{1}(s)>\dots
>\lambda_{n}(s)$ is given by (\ref{evd}) and that $\gamma$ denotes the
Lebesgue measure. Then, 
\begin{align}
\mathbb{E}\Big[| & \lambda_{i}(s)-\lambda_{j}(s)|^{-r}\Big]  \notag \\
& =c_{n,p}\int_{\mathcal{S}_{n}}\prod_{j=1}^{n}\left(
\lambda_{j}^{(p-n-1)/2}|\lambda_{i}-\lambda_{j}|^{-r}s^{-npH}\exp\left( -%
\frac{\lambda_{j}}{2s^{2H}}\right) \right)
\prod_{j<k}|\lambda_{k}-\lambda_{j}| d\gamma.  \notag
\end{align}
By making the change of variable $\lambda_{i}=\mu_{i}s^{H}$ and performing
the integration, we observe 
\begin{equation*}
\mathbb{E}\Big[|\lambda_{i}(s)-\lambda_{j}(s)|^{-r}\Big]\leq C_{p,n}s^{-2rH},
\end{equation*}
for a positive constant that depends on $p$ and $n$. In other words, $%
\mathbb{E}[|\lambda_{i}(s)-\lambda_{j}(s)|^{-r}]$ is uniformly bounded on
the interval $[t_{0},T]$. Therefore for all $i\not =j$, $%
\int_{t_{0}}^{T}|I_{i,j}(s)||J_{j}(s)|ds<\infty$, and so the claim (\ref%
{prob}) follows.

Finally, the identity (\ref{prob}) implies that $\mathbb{P}(T<\tau)=1$,
otherwise we would get a contradiction since $\log(\lambda_{i}(\tau
)-\lambda_{j}(\tau))=-\infty$. Therefore as $T$ goes to $\infty$, we obtain $%
\mathbb{P}(\tau=\infty)=1$. We obtain the desired result by letting $t_{0}$
go to zero.
\end{proof}

\section{Functional limit for the Fractional Wishart Process.}

For a probability measure $\mu$ and a $\mu$-integrable function $f$, we
write $\left\langle \mu,f\right\rangle =\int f(x)\mu(\mathrm{d}x).$ Hence,
since the empirical measure $\mu^{(n)}$ is a point measure, we have, for $%
f\in C_{b}^{2}$, that 
\begin{equation}
\langle\mu_{t}^{(n)},f\rangle=\frac{1}{n}\sum_{i=1}^{n}f(%
\lambda_{i}^{(n)}(t)).  \label{emf}
\end{equation}
Therefore, applying the chain rule to the last equation, we get 
\begin{equation}
\langle\mu_{t}^{(n)},f\rangle=\langle\mu_{0}^{(n)},f\rangle+\frac{1}{n}%
\sum_{i=1}^{n}\int_{0}^{t}f^{\prime}(\lambda_{i}^{(n)}(s))
d\lambda_{i}^{(n)}(s).  \label{1}
\end{equation}
In order to consider the dynamics of the measure-valued process $(\mu
_{t}^{(n)}, t\geq0)$, we prove the following result. Recall that the
fractional Wishart process $X^{(n)}$ is defined by 
\begin{equation*}
X^{(n)}(t)=(N^{(n)}(t))^{*}N^{(n)}(t),\qquad\text{for }\quad t\ge0,
\end{equation*}
where $N^{(n)}(t)=n^{-1/2} N(t)$ and $N$ denotes the matrix fractional
Brownian motion ($n\times p$).

\begin{lemma}
Let $(\mu_{t}^{(n)}, t\geq0)$ be the empirical measure-valued process of the
eigenvalues of the fractional Wishart process $(X^{(n)}(t), t\ge0)$. Then
for $f\in C_{b}^{2}(\mathbb{R})$, we have 
\begin{align}
\langle\mu_{t}^{(n)},f\rangle & =\langle\mu_{0}^{(n)},f\rangle+\frac {1}{%
n^{3/2}}\sum_{i=1}^{n}\sum_{k=1}^{p} \sum_{h=1}^{n} \int_{0}^{t}f^{\prime
}(\lambda_{i}^{(n)}(s))\frac{\partial\Phi_{i}}{\partial b_{kh}}%
(N^{(n)}(s))\delta b_{kh}(s)  \notag \\
& \hspace{1cm} +H\int_{0}^{t}\int_{\mathbb{R}^{2}}(f^{\prime}(x)-f^{\prime
}(y))\frac{x+y}{x-y}s^{2H-1}\mu^{(n)}_{s}(dx)\mu^{(n)}_{s}(dy)ds  \notag \\
& \hspace{2cm}+\frac{2Hp}{n}\int_{0}^{t}f^{\prime}(x)
s^{2H-1}\mu_{s}^{(n)}(dx) +\frac{2H}{n}\int_{0}^{t}\int_{\mathbb{R}%
}f^{\prime\prime}(x)xs^{2H-1}\mu _{s}^{(n)}(dx) ds.  \label{Ident1}
\end{align}
\end{lemma}

\begin{proof}
We first observe from \cite{AN} that we can apply It\^o's formula with
respect to the Young integral to the eigenvalues of the process $X^{(n)}$
and get 
\begin{equation}  \label{ito1}
\lambda_{i}^{(n)}(t)=\lambda_{i}^{(n)}(0)+\frac{1}{\sqrt{n}}%
\sum_{k=1}^{p}\sum_{h=1}^{n}\int_{0}^{t} \frac{\partial\Phi_{i}}{\partial
b_{kh}}(N^{(n)}(s)) d b_{kh}(s),
\end{equation}
for any $t\ge0$ and $i=\{1,2,\ldots, n\}$. Hence using (\ref{1}) and (\ref%
{ito1}), we obtain 
\begin{equation*}
\langle\mu_{t}^{(n)},f\rangle=\langle\mu_{0}^{(n)},f\rangle+\frac{1}{n^{3/2}}%
\sum_{i=1}^{n}\sum_{k=1}^{p}\sum_{h=1}^{n}\int_{0}^{t}f^{\prime}(\Phi
_{i}(N^{(n)}(s)))\frac{\partial\Phi_{i}}{\partial b_{kh}}%
(N^{(n)}(s))db_{kh}(s).
\end{equation*}
Now we will be interested in replacing the Young integrals by Skorokhod
integrals in the above expression. To this end, we prove that the condition
of Proposition 3 of \cite{EN} is satisfied. We will denote by $D^{kh}$ the
Malliavin derivative with respect to $b_{kh}$, for each $1\leq k\leq h\leq n$%
.

We will first show that 
\begin{equation}  \label{cond3.1}
\int_{0}^{t}\int_{0}^{t}D^{kh}_{r}\bigg(f^{\prime}(\Phi_{i}(N^{(n)}(s)))%
\frac{\partial\Phi_{i}}{\partial b_{kh}}(N^{(n)}(s))\bigg)%
|s-r|^{2H-2}drds<\infty,\qquad\text{$\mathbb{P}$-a.s.}
\end{equation}
In this direction, we first observe 
\begin{equation}  \label{Malder}
\begin{split}
\int_{0}^{t}\int_{0}^{t} & D_{r}^{kh}\bigg(f^{\prime}(\Phi_{i}(N^{(n)}(s)))%
\frac{\partial\Phi_{i}}{\partial b_{kh}}(N^{(n)}(s))\bigg)|s-r|^{2H-2}drds \\
& =\frac{1}{n(2H-1)}\int_{0}^{t}f^{\prime\prime}(\Phi_{i}(N^{(n)}(s)))\left( 
\frac{\partial\Phi_{i}}{\partial b_{kh}}(N^{(n)}(s))\right) ^{2}s^{2 H-1}ds
\\
& \hspace{3cm} +\frac{1}{n(2H-1)}\int_{0}^{t}f^{\prime}(\Phi_{i}(N^{(n)}(s)))%
\frac{\partial^{2}\Phi_{i}}{\partial b_{kh}^{2}}(N^{(n)}(s)))s^{2H-1}ds.
\end{split}%
\end{equation}
Now, using (\ref{ev5}) and the Cauchy--Schwarz inequality, 
\begin{equation}  \label{cuadra}
\bigg|\frac{\partial\Phi_i}{\partial b_{kh}}(N^{(n)}(s)))\bigg|^{2}=\frac{4}{%
n}\Big(U^{(n)}_{hi}(s)\Big)^{2}\left( \sum_{r=1}^{n}U^{(n)}_{r{i}%
}(s)b_{kr}(s)\right) ^{2}\leq\frac{4}{n}\sum_{r=1}^{n}b_{kr}^{2}(s),
\end{equation}
implying 
\begin{align}
\mathbb{E}\left[ \left|
\int_{0}^{t}f^{\prime\prime}(\Phi_{i}(N^{(n)}(s)))\left( \frac{%
\partial\Phi_{i}}{\partial b_{kh}}(N^{(n)}(s))\right) ^{2}s^{2H-1}ds\right| %
\right] & \leq\frac{4}{n}\|f^{\prime\prime}\|_{\infty}\left\vert
\int_{0}^{t}\sum_{r=1}^{n}\mathbb{E} \Big[ b_{kr}^{2}(s)\Big] %
s^{2H-1}ds\right\vert  \notag \\
& =\frac{t^{4H}}{H}\|f^{\prime\prime}\|_{\infty}<\infty.  \notag
\end{align}
On the other hand, using inequality (\ref{cee}), we obtain 
\begin{align}
\mathbb{E}\left[ \left\vert \int_{0}^{t}f^{\prime}(\Phi_{i}(N^{(n)}(s)))%
\frac{\partial^{2}\Phi_{i}}{\partial b_{kh}^{2}}(N^{(n)}(s))s^{2H-1}ds\right%
\vert \right] & \leq\Vert f^{\prime}\Vert_{\infty} \int_{0}^{t}\mathbb{E}%
\left[ \left\vert \frac{\partial^{2}\Phi_{i}}{\partial b_{kh}^{2}}%
(N^{(n)}(s))\right\vert \right] s^{2H-1}ds  \notag \\
& \leq\frac{\tilde{K}_{n,p,1}}{2H}t^{2H}<\infty,  \notag
\end{align}
thus, we conclude 
\begin{equation*}
\left\vert \int_{0}^{t}f^{\prime}(\Phi_{i}(N^{(n)}(s)))\frac{%
\partial^{2}\Phi_{i}}{\partial b_{kh}^{2}}(N^{(n)}(s))s^{2H-1}ds\right\vert
<\infty \qquad\text{$\mathbb{P}$-a.s.}
\end{equation*}
Putting the pieces together, we obtain that (\ref{cond3.1}) holds.

Now, we apply Proposition 3 of \cite{EN} (see also Proposition 5.2.3 of \cite%
{N}) in order to express the Young integrals that appear in (\ref{ito1}) in
terms of Skorokhod integrals. Therefore 
\begin{equation*}
\begin{split}
\langle\mu_{t}^{(n)},f\rangle & =\langle\mu_{0}^{(n)},f\rangle+\frac {1}{%
n^{3/2}}\sum_{i=1}^{n}\sum_{k=1}^{p}\sum_{h=1}^{n}\int_{0}^{t}f^{\prime
}(\Phi_{i}(N^{(n)}(s)))\frac{\partial\Phi_{i}}{\partial b_{kh}}%
(N^{(n)}(s))\delta b_{kh}(s) \\
& \hspace{0.5cm}+\frac{H(2H-1)}{n^{2}}\sum_{i=1}^{n}\sum_{k=1}^{p}\sum
_{h=1}^{n}\int_{0}^{t}\int_{0}^{t}D_{r}^{kh}\bigg(f^{\prime}(%
\Phi_{i}(N^{(n)}(s)))\frac{\partial\Phi_{i}}{\partial b_{kh}}(N^{(n)}(s))%
\bigg)|s-r|^{2H-2}drds. \\
&
\end{split}%
\end{equation*}
Next, we apply identity (\ref{Malder}) and deduce 
\begin{equation*}
\begin{split}
\langle\mu_{t}^{(n)},f\rangle & =\langle\mu_{0}^{(n)},f\rangle+\frac {1}{%
n^{3/2}}\sum_{i=1}^{n}\sum_{k=1}^{p}\sum_{h=1}^{n}\int_{0}^{t}f^{\prime
}(\Phi_{i}(N^{(n)}(s)))\frac{\partial\Phi_{i}}{\partial b_{kh}}%
(N^{(n)}(s))\delta b_{kh}(s) \\
& \hspace{2cm}+\frac{H}{n^{2}}\sum_{i=1}^{n}\sum_{k=1}^{p}\sum_{h=1}^{n}%
\int_{0}^{t}f^{\prime\prime}(\Phi_{i}(N^{(n)}(s)))\left( \frac{\partial
\Phi_{i}}{\partial b_{kh}}(N^{(n)}(s))\right) ^{2}s^{2H-1}ds \\
& \hspace{3cm}+\frac{H}{n^{2}}\sum_{i=1}^{n}\sum_{k=1}^{p}\sum_{h=1}^{n}%
\int_{0}^{t}f^{\prime}(\Phi_{i}(N^{(n)}(s)))\frac{\partial^{2}\Phi_{i}}{%
\partial b_{kh}^{2}}(N^{(n)}(s)))s^{2H-1}ds. \\
&
\end{split}%
\end{equation*}
Since $U^{(n)}$ is an orthogonal matrix, we observe from (\ref{cuadra}) and (%
\ref{ev3}) that 
\begin{equation*}
\sum_{k=1}^{p}\sum_{h=1}^{n}\bigg(\frac{\partial\Phi_i}{\partial b_{kh}}%
(N^{(n)}(s)))\bigg)^{2}=4\sum_{k=1}^{p}\left(
\sum_{r=1}^{n}U_{r_{i}}^{(n)}(s)b^{(n)}_{kr}(s)\right)
^{2}=4\lambda_{i}^{(n)}(s).
\end{equation*}
Therefore from the latter identity and applying (\ref{sd}), we deduce 
\begin{equation*}
\begin{split}
\langle\mu_{t}^{(n)},f\rangle & =\langle\mu_{0}^{(n)},f\rangle+\frac {1}{%
n^{3/2}}\sum_{i=1}^{n}\sum_{k=1}^{p}\sum_{h=1}^{n}\int_{0}^{t}f^{\prime
}(\Phi_{i}(N^{(n)}(s)))\frac{\partial\Phi_{i}}{\partial b_{kh}}%
(N^{(n)}(s))\delta b_{kh}(s) \\
& \hspace{2cm}+\frac{2H}{n^{2}}\sum_{i=1}^{n}\int_{0}^{t}f^{\prime}(%
\Phi_{i}(N^{(n)}(s)))\left( p+\sum_{j\not =i}\frac{%
\lambda_{i}^{(n)}(s)+\lambda_{j}^{(n)}(s)}{\lambda_{i}^{(n)}(s)-%
\lambda_{j}^{(n)}(s)}\right) s^{2H-1}ds \\
& \hspace{5cm}+\frac{4H}{n^{2}}\sum_{i=1}^{n}\int_{0}^{t}f^{\prime\prime
}(\Phi_{i}(N^{(n)}(s)))\lambda_{i}^{(n)}(s)s^{2H-1}ds \\
& =\langle\mu_{0}^{(n)},f\rangle+\frac{1}{n^{3/2}}\sum_{i=1}^{n}\sum
_{k=1}^{p}\sum_{h=1}^{n}\int_{0}^{t}f^{\prime}(\Phi_{i}(N^{(n)}(s)))\frac {%
\partial\Phi_{i}}{\partial b_{kh}}(N^{(n)}(s))\delta b_{kh}(s) \\
& \hspace{1cm}+H\int_{0}^{t}\int_{\mathbb{R}^{2}}(f^{\prime}(x)-f^{\prime
}(y))\frac{x+y}{x-y}s^{2H-1}\mu_{s}^{(n)}(dx)\mu_{s}^{(n)}(dy)ds \\
& \hspace{2cm}+\frac{2Hp}{n}\int_{0}^{t}\int_{\mathbb{R}}f^{\prime
}(x)s^{2H-1}\mu_{s}^{(n)}(dx)ds+\frac{4H}{n}\int_{0}^{t}\int_{%
\mathbb{R}}f^{\prime\prime}(x)xs^{2H-1}\mu_{s}^{(n)}(dx)ds,
\end{split}%
\end{equation*}
where in the last identity we used the definition of $\mu^{(n)}$. This
completes the proof.
\end{proof}

\subsection{Tightness}

\label{Tightness}

In this section, we prove that the family of measures $\{(\mu _{t}^{(n)},
t\ge0):n\geq1\}$ is tight. For this purpose, we first prove the following
auxiliary result.

\begin{lemma}
Let $(\lambda_{1}^{n}(t),\dots,\lambda_{n}^{(n)}(t); t\ge0)$ be the
eigenvalues of the fractional Wishart process $X^{(n)}$ with parameter $%
H\in(1/2,1)$. Then for all $s,t\in[0,T],$ we have 
\begin{equation}  \label{esthold}
\mathbb{E}\left[ \left( \frac{1}{n}\sum_{i=1}^{n}|\lambda
_{i}^{(n)}(t)-\lambda_{i}^{(n)}(s)|\right) ^{4}\right] \leq\left( \frac {p}{n%
}\right) ^{4}C|t-s|^{4H}T^{4H},
\end{equation}
where $C$ is a positive constant.
\end{lemma}

\begin{proof}
As in the proof of Theorem 4, we can use the Hoffman--Wielandt
inequality (see \cite{HW}) to deduce
\begin{align}
\sum_{i=1}^{n}\Big(\lambda_{i}^{(n)}(t)-\lambda^{(n)}_{i}(s)\Big)^{2}\leq
\sum_{i,j=1}^{n}\left( \sum_{k=1}^{p}\left(
b_{ki}^{(n)}(t)b^{(n)}_{kj}(t)-b^{(n)}_{ki}(s)b^{(n)}_{kj}(s)\right) \right)
^{2}.  \notag
\end{align}%
 From similar arguments as those used at the beginning of the proof of
Theorem 4, we get that there exists a constant $C>0$ such that for any $r\in\mathbb{N}$ 
\begin{align*}
\mathbb{E}\Big[ \Big(%
b^{(n)}_{ki}(t)b^{(n)}_{kj}(t)-b^{(n)}_{ki}(s)b^{(n)}_{kj}(s)\Big)^{r}\Big] %
\leq\frac{C}{n^{r}}|t-s|^{rH}T^{rH}  \notag
\end{align*}
and, 
\begin{align*}
\mathbb{E}\Big[ \Big((b^{(n)}_{ki}(t))^2-(b^{(n)}_{ki}(s))^2\Big)^{r}\Big] %
\leq\frac{C}{n^{2r}}|t-s|^{2rH}T^{2rH}
\end{align*}%
Therefore using the Cauchy--Schwarz inequality, Jensen's inequality twice,
and the previous inequality, we obtain that the exists a
constant $K>0$, that depends only on $T$, such that 
\begin{align}
\mathbb{E}\left[ \left( \frac{1}{n}\sum_{i=1}^{n}\big|\lambda_{i}^{(n)}(t)-%
\lambda_{i}^{(n)}(s)\big|\right) ^{4}\right] & \leq\frac {1}{n^{2}}\mathbb{E}%
\left[ \left( \sum_{i=1}^{n}\big|\lambda_{i}^{(n)}(t)-\lambda_{i}^{(n)}(s)%
\big|^{2}\right) ^{2}\right]  \notag \\
& \leq\frac{1}{n^{2}}\mathbb{E}\left[ \left( \sum_{i,j=1}^{n}\left(
\sum_{k=1}^{p}\left(
b_{ki}^{(n)}(t)b^{(n)}_{kj}(t)-b^{(n)}_{ki}(s)b^{(n)}_{kj}(s)\right) \right)
^{2}\right) ^{2}\right]  \notag \\
& \leq\mathbb{E}\left[ \sum_{i,j=1}^{n}\left( \sum_{k=1}^{p}\left(
b_{ki}^{(n)}(t)b^{(n)}_{kj}(t)-b^{(n)}_{ki}(s)b^{(n)}_{kj}(s)\right) \right)
^{4}\right]  \notag \\
& \leq K \left( \frac{p}{n}\right) ^{4} |t-s|^{4H}T^{4H}.  \notag
\end{align}%
This completes the proof.
\end{proof}

\begin{theorem}
\label{teoT} Assume that $p:=p(n)$ is such that $p/n\to c$, as $n\to\infty$.
Then the family of measures $\{(\mu_{t}^{(n)}, t\geq0); n\geq1\}$ is tight.
\end{theorem}

\begin{proof}
It is easily seen using (\ref{emf}) that for every $0\leq t_{1}\leq
t_{2}\leq T$, $n\geq1$ and $f\in\mathcal{C}^{2}_{b}$, 
\begin{align}  \label{est0}
\Big|\langle\mu^{(n)}_{t_{2}}, f\rangle-\langle\mu^{(n)}_{t_{1}},f\rangle%
\Big| & \leq\frac{1}{n}\sum_{i=1}^{n}\Big|f(\lambda_{i}^{(n)}(t_{2}))-f(%
\lambda_{i}^{(n)}(t_{1}))\Big|.
\end{align}
On the other hand by (\ref{evd2}), we know that for each $n\geq1$, the
functions $\lambda_{i}^{(n)}$ are H\"older continuous of order $\beta<H$.
Therefore, since $f^{\prime}$ is bounded and applying the Mean Value
Theorem, we deduce 
\begin{equation*}
\Big|f(\lambda_{i}^{(n)}(r))-f(\lambda_{i}^{(n)}(s))\Big|\leq\|f^{\prime}%
\|_{\infty}\Big|\lambda_{i}^{(n)}(r)-\lambda _{i}^{(n)}(s)\Big|.
\end{equation*}
Hence using the above estimate, identity (\ref{esthold}), and Jensen's
inequality, we obtain 
\begin{align}  \label{est2}
\mathbb{E}\Bigg[\Big|\langle\mu^{(n)}_{t_{2}}, f\rangle-\langle
\mu^{(n)}_{t_{1}},f\rangle\Big|^{4}\Bigg] & \leq\mathbb{E}\left[ \left( 
\frac{1}{n}\sum_{i=1}^{n}\Big|f(\lambda_{i}^{(n)}(t_{2}))-f(%
\lambda_{i}^{(n)}(t_{1}))\Big|\right) ^{4}\right]  \notag \\
& \leq\|f^{\prime}\|_{\infty}^{4} \mathbb{E}\left[ \left( \frac{1}{n}%
\sum_{i=1}^{n}|\lambda_{i}^{(n)}(t_{2})-\lambda_{i}^{(n)}(t_{1})|\right) ^{4}%
\right]  \notag \\
& \leq C_{f,T}|t_{1}-t_{2}|^{4H},
\end{align}
where in the last inequality we used our assumption and that $C_{f,T}$ is a
constant that depends on $f^{\prime}$ and $T$.

Therefore, by the well known criterion that appears in \cite{EK} (see Prop.
2.4), we have that the sequence of continuous real processes $\{(\langle
\mu_{t}^{(n)},f\rangle,t\geq0) ; n\geq1\}$ is tight and consequently the
sequence of processes $\{(\mu_{t}^{(n)}, t\geq0); n\geq1\}$ is tight on $C(%
\mathbb{R}_{+},\mathrm{Pr}(\mathbb{R}))$.
\end{proof}

{\color{blue} }

\subsection{Weak convergence of the empirical measure of eigenvalues}

In the previous section, we proved that the family of measures $\{(\mu
_{t}^{(n)},t\geq 0);n\geq 1\}$ is tight on $C(\mathbb{R}_{+},\mathrm{Pr}(%
\mathbb{R}))$, which allows us to prove now our main result. Before that, we
first characterize the family of laws $(\mu _{c,H}(t),t\geq 0)$ of
fractional dilations of a free Poisson distribution in terms of the initial
value problem of the corresponding Cauchy--Stieltjes transforms. We note
that the case $H=1/2$ was proved in Corollary 3.1 of \cite{CG}.

\begin{proposition}
\label{IVP} The family $(\mu_{c,H}(t), t\geq0)$ is characterized by the
property that its Cauchy--Stieltjes transform is the unique solution to the
initial value problem 
\begin{equation}
\begin{cases}
\displaystyle\frac{\partial G_{c,H}}{\partial t}(t,z)=2H\left[
G_{c,H}^{2}(t,z)+\Big[1-c+2zG_{c,1/2}(t,z)\Big]\frac{\partial G_{c,H}}{%
\partial z}(t,z)\right] t^{2H-1}, & \qquad\text{$t>0$,} \\ 
\displaystyle G_{c,H}(0,z)=-\frac{1}{z}, & \qquad\text{$z\in\mathbb{C}^{+}$},%
\end{cases}
\label{ivp}
\end{equation}
satisfying $G_{t}(z)\in\mathbb{C}^{+}$ for $z\in\mathbb{C}^{+}$ and 
\begin{equation*}
\lim_{\eta\rightarrow\infty}\eta|G_{t}(i\eta)|<\infty,\qquad\text{for each }
\quad t>0.
\end{equation*}
\end{proposition}

\begin{proof}
Recall from Section 1 that the family of fractional dilations of a free
Poisson distribution is such that for each $t>0$, $\mu_{c,H}(t)=\mu
_{c}^{f,p}\circ(h_{t}^{H})^{-1},$ where $\mu_{c}^{f,p}$ is the free Poisson
distribution and $h_{t}^{H}(x)=t^{2H}x$.

Therefore, the Cauchy--Stieltjes transform $G_{c,H}(t,z)$ of the
distribution $\mu _{c,H}(t)$ satisfies 
\begin{align}
G_{c,H}(t,z)=\int_{\mathbb{R}}\frac{\mu_{c,H}(t)(dx)}{x-z}=\int_{\mathbb{R}}%
\frac{\mu_{c}^{f,p}\circ(h_{t}^{H})^{-1} (dx)}{x-z}=\int_{\mathbb{R}}\frac{%
\mu_{c}^{f,p}(dx)}{xt^{2H}-z}=\frac{1}{t^{2H}}\int_{\mathbb{R}}\frac {%
\mu_{c}^{f,p}(dx)}{x-zt^{-2H}},  \notag
\end{align}
for all $z$ with Im($z$)$\not =0$. The above equality implies that 
\begin{equation*}
G_{c,H}(t,z)=\frac{1}{t^{2H}}G_{c}^{f,p}(zt^{-2H})=G_{c,1/2}(t^{2H},z),
\end{equation*}
where $G_{c}^{f,p}$ is the Cauchy--Stieltjes transform of the free Poisson
distribution.

If we assume that $\mu_{c,H}(0)=\mu_{c,1/2}(0)$, then by Corollary 3.1 in 
\cite{CG} (see also Proposition 2.1 of \cite{AT}), we know that $G_{c,1/2}$
is the unique solution to the initial value problem 
\begin{equation}
\begin{cases}
\displaystyle \frac{\partial G_{c,1/2}}{\partial t}(t,z)=G^{2}_{c,1/2}(t,z)+%
\Big[1-c+2zG_{c,1/2}(t,z)\Big]\frac{\partial G_{c,1/2}}{\partial z}(t,z), & 
\qquad\text{$t>0$,}\notag \\ 
\displaystyle G_{c, 1/2}(0,z)=-\frac{1}{z}, & \qquad\text{$z\in\mathbb{C}%
^{+} $}.\notag%
\end{cases}%
\end{equation}
Therefore 
\begin{align}
\frac{\partial G_{c,H}}{\partial t}(t,z) & =\frac{\partial G_{c,1/2}}{%
\partial t}(t^{2H},z)=2Ht^{2H-1}\frac{\partial G_{c,1/2}}{\partial t}(t,z)%
\bigg|_{(t,z)=(t^{2H},z)}  \notag \\
& =2H\left[ G^{2}_{c,1/2}(t^{2H},z)+[1-c+2zG_{c,1/2}(t^{2H},z)]\frac {%
\partial G_{c,1/2}}{\partial z}(t^{2H},z)\right] t^{2H-1}  \notag \\
& =2H\left[ G^{2}_{c,H}(t,z)+\Big[1-c+2zG_{c,H}(t,z)\Big]\frac{\partial
G_{c,H}}{\partial z}(t,z)\right] t^{2H-1}.  \notag
\end{align}
On the other hand, note that 
\begin{equation*}
G_{c,H}(0,z)=\int_{\mathbb{R}}\frac{\mu_{c,H}(0)(dx)}{x-z}.
\end{equation*}
Finally, the uniqueness of (\ref{ivp}) follows from Corollary 3.1 of \cite%
{CG}.
\end{proof}

Now, we are ready to prove our main result, i.e., that the weak limit of the
sequence of the measure-valued processes $\{(\mu_{t}^{(n)}, t\ge0); n\geq1\} 
$ converges to the family of fractional dilations of the free Poisson
distribution $\{\mu_{c,H}(t);t\geq0\}$.

\begin{proof}[Proof of Theorem \protect\ref{thcfw}]
From Theorem \ref{teoT}, we know that the family $\{(\mu_{t}^{(n)}, t\ge0);
n\geq1\}$ is relatively compact. Hence, for our purposes, we take a
subsequence $\{(\mu_{t}^{(n_{\ell})}, t\ge0); \ell\geq1\}$ and we assume
that it converges weakly to $(\mu_{t}, t\ge0)$, with $\mu_0=\delta_0$.

Now we will consider the Cauchy--Stieltjes transform $(G_{t}^{n_{l}},t\geq
0) $ of $(\mu _{t}^{(n_{l})},{t\geq 0})$ given by 
\begin{equation*}
G_{t}^{n_{l}}(z)=\int_{\mathbb{R}}\frac{\mu _{t}^{(n_{l})}(dx)}{x-z}.
\end{equation*}%
For fixed $t\geq 0$, let us consider the eigenvalues $\{\lambda
^{(n_{l})}(t)\}_{i=1}^{n}$ of the matrix $X^{(n_{l})}(t)$. We observe that $%
X^{(n_{l})}(t)$ and $t^{2H}X^{(n_{l})}(1)$ have the same distribution, thus $%
\{\lambda ^{(n_{l})}(t)\}_{i=1}^{n_{l}}$ and $t^{2H}\{\lambda
^{(n_{l})}(1)\}_{i=1}^{n_{l}}$ also have the same distribution. This implies
that 
\begin{equation*}
G_{t}^{n_{l}}(z)=\int_{\mathbb{R}}\frac{\mu _{t}^{(n_{l})}(dx)}{x-z}=\int_{%
\mathbb{R}}\frac{\mu _{1}^{(n_{l})}(dx)}{t^{2H}x-z}=\frac{1}{t^{2H}}\int_{%
\mathbb{R}}\frac{\mu _{1}^{(n_{l})}(dx)}{x-t^{-2H}z}=\frac{1}{t^{2H}}%
G_{1}^{n_{l}}(t^{-2H}z).
\end{equation*}%
Now by the result of Marchenko and Pastur \cite{MP} we know that $\mu
_{1}^{(n_{l})}$ converges weakly almost surely to the free Poisson
distribution $\mu ^{fp}$ which has Cauchy--Stieltjes transform $G^{fp}$
given by 
\begin{equation*}
G^{fp}(z)=\frac{-(z+1-c)+\sqrt{(z+1-c)^{2}-4c}}{2z}.
\end{equation*}%
Hence we obtain that the Cauchy-Stieltjes transform $G_{t}$ of $\mu _{t}$ is
given by 
\begin{align*}
G_{t}(z)=\lim_{l\rightarrow \infty }G_{t}^{n_{l}}(z)=\lim_{l\rightarrow
\infty }\frac{1}{t^{2H}}G_{1}^{n_{l}}(t^{-2H}z)& =\frac{1}{t^{2H}}\frac{%
-(t^{-2H}z+1-c)+\sqrt{(t^{-2H}z+1-c)^{2}-4c}}{2zt^{-2H}} \\
& =\frac{-(z+(1-c)t^{2H})+\sqrt{(z+(1-c)t^{2H})^{2}-4ct^{4H}}}{2zt^{2H}}.
\end{align*}%
Using the above identity and making some straightforward computations, it is
easy to verify that $(G_{t},t\geq 0)$ satisfies (\ref{ivp}), and therefore
the family $(\mu _{t},t\geq 0)$ corresponds to the family of fractional
dilations of a free Poisson distribution.

Therefore, we conclude that all limits of subsequences of $(\mu _{t}^{(n)},{%
t\geq 0})$ coincide with the family $(\mu _{t},{t\geq 0})$, with its
Cauchy--Stieltjes transform given as the solution to (\ref{ivp}), and thus
the sequence $\{(\mu _{t}^{(n)})_{t\geq 0}:n\geq 1\}$ converges weakly to $%
(\mu _{t},{t\geq 0})$.
\end{proof}

\begin{remark}
A more general version of Theorem 1 can be obtained in the case when $\mu
_{0}^{n}$ converges weakly to a measure $\mu _{0}$ which is not necessarily $%
\delta _{0}$.

This is done with a more detailed analysis of equation \eqref{Ident1} when
applied to the deterministic sequence of functions 
\begin{equation*}
f_{j}(x)=\frac{1}{x-z_{j}},\qquad \text{$z_{j}\in (\mathbb{Q}\times \mathbb{Q%
})\cap \mathbb{C}^{+}$,}
\end{equation*}%
and using a continuity argument. It can be proven that the Cauchy--Stieltjes
transform $(G_{t},t\geq 0)$ of $\ (\mu _{t},{t\geq 0})$ satisfies the
integral equation 
\begin{align}
G_{t}(z)& =\int_{\mathbb{R}}\frac{\mu _{0}(dx)}{x-z}+H\int_{0}^{t}\int_{%
\mathbb{R}^{2}}\left( \frac{1}{(y-z)^{2}}-\frac{1}{(x-z)^{2}}\right) \frac{%
x+y}{x-y}s^{2H-1}\mu _{s}(dx)\mu _{s}(dy)ds  \notag  \label{cst} \\
& \hspace{1cm}-2Hc\int_{0}^{t}\int_{\mathbb{R}}\frac{1}{(x-z)^{2}}%
s^{2H-1}\mu _{s}(dx).
\end{align}%
After some computations it is not difficult to see that the
Cauchy--Stieltjes transform $(G_{t},t\geq 0)$ of $(\mu _{t},{t\geq 0})$ is
the unique solution to the initial value problem given by 
\begin{equation}
\begin{cases}
\displaystyle\frac{\partial G_{c,1/2}}{\partial t}(t,z)=G_{c,1/2}^{2}(t,z)+%
\Big[1-c+2zG_{c,1/2}(t,z)\Big]\frac{\partial G_{c,1/2}}{\partial z}(t,z), & 
\qquad \text{$t>0$,}\notag \\ 
\displaystyle G_{c,1/2}(0,z)=\int_{\mathbb{R}}\frac{\mu _{0}(dx)}{x-z}, & 
\qquad \text{$z\in \mathbb{C}^{+}$}.\notag%
\end{cases}%
\end{equation}
\end{remark}

\textbf{Acknowledgement}. The authors would like to thank the anonymous
referee for the comments and suggestions provided that improved the
presentation of the paper.

\end{document}